\documentclass[preprint,authoryear,12pt]{article}
\usepackage{newlfont}
\usepackage[utf8]{inputenc}
\usepackage[T1]{fontenc}
\usepackage{a4wide}
\usepackage{amsmath}
\usepackage{inputenc}
\usepackage{amsfonts}
\usepackage{amssymb}
\usepackage{amsthm}
\usepackage{newlfont}
\usepackage{graphicx}
\usepackage{subfigure}
\usepackage{mathrsfs}
\usepackage[authoryear]{natbib}
\usepackage{color}

\usepackage{comment}

\usepackage[running]{lineno}

\newtheorem{theorem}{Theorem}
\newtheorem{lemma}{Lemma}

\newtheorem{corollary}{Corollary}

\newtheorem{definition}{Definition\rm}
\newtheorem{remark}{Remark}

\theoremstyle{definition}
\newtheorem{example}{Example}

\renewcommand{\AA}{{\mathcal A}}

\newcommand{\FF}{{\mathcal F}}

\newcommand{\IE}{{\mathrm{I\!E}}}

\newcommand{\LL}{{\mathcal L}}

\newcommand{\HH}{{\mathcal H}}

\newcommand{\co}{\mathop{\mathrm{co}}}

\newcommand{\EssOsc}{\mathop{\mathrm{EssOsc}}}

\begin{document}


\title{Remark on norm compactness in $L^p(\mu,X)$.\footnote{This is a preliminary version.}}
\author{Y. Askoura}

\maketitle

{Lemma, Universit\'e Paris II, 
4 rue Blaise Desgoffe, 75006 Paris.}

{youcef.askoura@u-paris2.fr}

\begin{abstract} We prove a compactness criterion in $L^p(\mu,X)$: a subset of $L^p(\mu,X)$ is relatively norm compact iff the set of integrals of its functions over any measurable set is relatively norm compact, it satisfies the Fr\'echet oscillation restriction condition and it is $p$-uniformly integrable. The proof is elementary. 
\end{abstract}



\section{Introduction}
Relative norm compactness in spaces of Lebesgue-Bochner integrable functions is characterized by tightness together with a concentration condition (reduction of oscillations) in \cite{SIM87,AUB63,LIO61} for functions defined on an interval $[0,\; T]$. \citet{RoS03} provides some generalization of these results steel for functions defined on intervals.  Note that in \cite{SIM87} the tightness is expressed by assuming that some integrals must belong to some compact set.  The general results we know, proved by \citet{DiM99} (see \citet{NEE14,NEE07} for a different proof and \cite{BGJ94} for the particular case $L^1(\mu,X)$), characterize norm compactness through tightness, scalar relative compactness and the $p$-uniform integrability. The scalar relative compactness is obtained by Bocce criterion~: a reduction of oscillations. 

\medskip
Thereafter, we adopt an ``integral" tightness. Instead of assuming tightness, that is~: except on an arbitrarily small measurable set, the values of the considered functions belong uniformly to a compact set,  we assume that their integrals over any measurable set belong to a compact set. We prove a compactness criterion as described in the abstract.  Comparatively to \citet{DiM99} results, the Bocce criterion is replaced by the Fr\'echet one, by modifying slightly the tightness notion used in \citet{DiM99}.

\section{Main results}

Let $(\Omega,\AA,\mu)$ be a finite measure space, where $\AA$ is a $\sigma$-algebra and $\mu$ a countably additive positive and finite measure. The set $\AA^+$ refers to the set of measurable subsets $E\in \AA$ with $\mu(E)>0$. We do not make a difference between two measurable sets with a $\mu$-null symmetric measure. 

Consider a Banach space $X$ normed by $\|\cdot\|$.  Denote by $L^p(\mu, X)$, $1\leq p<+\infty$, the Lebesgue-Bochner space consisting of all equivalence classes of  (strongly) $\mu$-measurable \citep{DiU77}, $\mu$-a.e. equal functions $f$ defined from $\Omega$ to $X$ such that $\|f\|^p$ is $\mu$-integrable. The usual norm in $L^p(\mu, X)$ is denoted by $\|\cdot\|_p$. The set $\Pi$ consisting in all finite partitions $\pi\subset \AA^+$ of $\Omega$ is directed by refinement.  Given $\pi\in \Pi$, define on $L^1(\mu, X)$ the conditional expectation operator 
$$\IE_\pi(f)(\cdot)=\underset{K\in \pi}\sum \frac{1}{\mu(K)}\int_K f d\mu\chi_K(\cdot),$$ where $\chi_K(\cdot)$ stands for the characteristic function of $K$. 

Recall that a subset $\HH\subset L^p(\mu,X)$, $1\leq p<+\infty$,  is said to be $p$-uniformly integrable iff, the set $\{\|f\|^p : f\in\HH\}$ is uniformly integrable. That is, 
$$\underset{M\rightarrow +\infty}{\lim}\int_{\|f\|^p>M} \|f\|^pd\mu=0,\text{ uniformly in }f\in \HH.$$ 
Equivalently, $\HH$ is bounded in $ L_p(\mu,X)$ and $\underset{\mu(E)\rightarrow 0}{\lim}\int_E \|f\|^pd\mu=0$, uniformly in $f\in \HH$. 
For $p>1$ and a finite $\mu$, it results straightforwardly from H\"older's inequality that a $p$-uniformly integrable set is ($1$-)uniformly integrable. 

\begin{definition}\label{IntT}
A subset $\HH\subset L^p(\mu,X)$ is said to be \emph{integral tight} iff, for every $E\in \AA,\left\{\int_E f d\mu : f\in \HH\right\}$ is relatively norm compact.
\end{definition}
The following lemma is a straightforward vector valued version of Riesz Theorem. Its proof is a slight modification of the classical one \citep{BOG07} by adding the integral tightness condition.

\begin{lemma}\label{RieszV} A subset $\HH$ of $L^p(\mu,X)$, $1\leq p< +\infty$, is relatively norm compact iff, 
\begin{itemize}
\item[1)] $\HH$ is integral tight and, 
\item[2)] $\underset{\pi}{\lim}\; \underset{f\in \HH}{\sup} \|\IE_\pi (f)-f\|_p=0.$
\end{itemize}
\end{lemma}

\begin{proof} $\Rightarrow $) Let $A\in \AA^+$ and $f\in L^p(\mu,X)$. Remark using H\"older's inequality that 
$$\left(\int_A \|f\|d\mu\right)^p\leq \mu(A)^{p-1}\int_A\|f\|^p d\mu.$$ This provides
\begin{itemize}
\item[(a)] $\| \IE_\pi(f)\|_p\leq \|f\|_p$, and 
\item[(b)] the operator 
$$\begin{array}{rl}
\Lambda_A:L^p(\mu,X)&\rightarrow X\\
f&\mapsto \int_A f d\mu
\end{array}
$$
 is continuous for the norm topologies. 
\end{itemize}
Observe that 1) results immediately from (b).  For 2), let $\varepsilon>0$ and consider a finite covering of $\HH$ by open balls $B(f_i,\varepsilon/4)$ of radius $\varepsilon/4$. Since the set of simple functions is dense in $L^p(\mu,X)$, consider for every $i$ a simple function $\varphi_i$ such that $\|f_i-\varphi_i\|_p\leq \varepsilon/4$. Let $\pi\subset \AA^+$ a finite partition of $\Omega$ such that every $\varphi_i$ is constant on every element of $\pi$. Observe that $\IE_\pi(\varphi_i)=\varphi_i$ for all $i$.  Then, for every $f\in \HH$, there is some $i$ such that

$$\begin{array}{rl}
\|\IE_\pi (f)-f\|_p &\leq \|\IE_\pi (f)-\varphi_i\|_p+\|\varphi_i-f_i\|_p+\|f_i-f\|_p\\
&\leq \|\IE_\pi (f-\varphi_i)\|_p+\varepsilon/2\\
&\leq \|f-\varphi_i\|_p+\varepsilon/2\\
&\leq \varepsilon.
\end{array}
$$
It is clear that this formula remains valid for every $\pi'\geq \pi$. Then 2) is true as well. 

\medskip 
$\Leftarrow)$ Conversely, 1) implies that for every $\pi\in \Pi$, $\IE_\pi(\HH)$ is relatively compact as homeomorphic to a relatively compact subset of the product of $\Lambda_A(\HH), A\in \pi$, consisting in relatively compact sets. From 2), for every $\varepsilon>0$, there is $\pi\in \Pi$ such that $\underset{f\in \HH}{\sup} \|\IE_\pi (f)-f\|_p\leq \varepsilon/3$. Let $B(\IE_\pi(f_i), \varepsilon/3), i=1,..,n$ a finite covering of $\IE_\pi(\HH)$, by open balls of radius $\varepsilon/3$. That is, for every $f\in \HH$, $\|\IE_\pi(f)-\IE_\pi(f_i)\|\leq \varepsilon/3$ for some $i\in\{1,...,n\}$. Then, $\HH$ is covered by a finite number of open balls $B(f_i,\varepsilon), i=1,..,n$ of radius $\varepsilon$. Hence it is totally bounded. 
\end{proof}

Condition 2) above is a concentration or an oscillation reduction condition. Recall that the essential oscillation of a function $f$ on a set $A$ is~:   

$$\EssOsc (f,A)=\inf\{\sup \|f(\omega)-f(\omega')\| : \omega,\omega'\in A\setminus B : \mu(B)=0\}.$$

\begin{definition}[\citep{SaV95,DiM99}]A subset $\HH$ of $\mu$-measurable functions satisfies
\emph{Fr\'echet oscillation condition  ($\FF$)} iff, for every $\varepsilon>0$, there is a finite partition $\pi\in \Pi$, there is a set $\Omega_f\in \AA$ for every $f\in \HH$, such that $\mu(\Omega\setminus \Omega_f)\leq \varepsilon$, and $\EssOsc(f, \Omega_f\cap A)\leq \varepsilon$ for all $A\in \pi$. 
\end{definition}
 
Let us denote the mean of a function $f\in L_1(\mu,X)$ on $A\in \AA^+$ by $m_A(f)=\frac{1}{\mu(A)}\int_Af\;d\mu$. Then, a simple application of the mean value Theorem for Bochner integral \citep{DiU77}, allows to see that for every $A\in \AA^+$, and $f\in L_1(\mu,X)$, 
$$\|f(\omega)-m_A(f)\|\leq \EssOsc(f,A), \mu\text{-a.e. on  }A.$$

Indeed, $m_A(f)=\frac{1}{\mu(A)}\int_Af\;d\mu\in \overline{\co} f(A\setminus B)$, for all $B\subset A$, $B\in \AA$ with $\mu(B)=0$. Let $\varepsilon>0$ be fixed. For every $B\subset A$, $B\in \AA$, $\mu(B)=0$, there is a finite convex combination $\sum \alpha_k f(w'_k)$, $w'_k\in A\setminus B$, such that 
$$\forall\omega\in \Omega, \|f(\omega)-m_A(f)\|\leq \|f(\omega)-\sum \alpha_k f(w'_k)\|+\varepsilon\leq \max_k\|f(\omega)-f(w'_k)\|+\varepsilon.$$
Let $B_\varepsilon\subset A$, $B_\varepsilon\in \AA$ with $\mu(B_\varepsilon)=0$,  such that 
$$\EssOsc(f,A)> \sup\{\|f(\omega)-f(\omega')\| : \omega,\omega'\in A\setminus B_\varepsilon\}-\varepsilon.$$ Hence, for all $\omega\in A\setminus B_\varepsilon$, $\|f(\omega)-m_A(f)\|\leq \EssOsc(f,A)+2\varepsilon$. By considering a sequence $\varepsilon_n$ decreasing to $0$, the inequality holds on $A\setminus \cup B_{\varepsilon_n}$.

\begin{theorem}\label{VecV}
A set $\HH\subset L^p(\mu,X)$ is relatively norm compact iff, it satisfies the Fr\'echet oscillation condition ($\FF$), it is integral tight and $p$-uniformly integrable. 
\end{theorem}
\begin{proof}In the following we omit the symbol "$d\mu$" under the integral sign and denote simply $\int_A f$ instead of $\int_A fd\mu$. All the integrals are related to the measure $\mu$.
 
$\Rightarrow)$
The  $p$-uniform integrability of $\HH$ is  obvious (results for instance from the obvious  relative weak compactness of the set  $\{\omega\mapsto \|f(\omega)\|^p :f\in \HH\}$ in $L^1(\mu)$). The integral tightness follows from Lemma \ref{RieszV}. Observe now that we can assume without loss of generality that the functions of $\HH$ take their values in a separable complete metric space. Indeed, since $\HH$ is relatively norm compact, it is relatively compact for the topology of convergence in measure. Then it is tight \cite{DiM98}. That is, for every $n\geq 1$, there exists a compact subset $D_n$ of $X$ such that $\mu(\{\omega\in \Omega : f(\omega)\notin D_n\})\leq 1/n$, for every $f\in \HH$. Hence, up to modifying the functions of $\HH$ on a $\mu$-null set, we can assume that all the functions of $\HH$ take their values on the closure of a the $\sigma$-compact set $\underset{n\geq 1}{\cup} D_n$ which is separable and complete. Hence, we can apply the Fr\'echet Theorem (\cite{SaV95}, p. 425) to affirm that $\HH$ satisfies ($\FF$). 

\medskip
$\Leftarrow$) 
We have to prove that Condition 2) of Lemma \ref{RieszV} is satisfied. Then, the conclusion follows from Lemma \ref{RieszV}. 

\medskip
Let $\varepsilon>0$ be fixed. We can assume without loss of generality that $\varepsilon\leq1$. The $p$-uniform integrability of $\HH$ implies its uniform integrability. We know that the uniform integrability of $\HH$ implies the uniform integrability of the set of conditional expectations $\{\IE_\pi(\|f\|): f\in \HH, \pi\in \Pi\}$. For the reader's convenience, by Markov's inequality 
$$\mu(\IE_\pi(\|f\|)>M)\leq \frac{1}{M}\int_\Omega \IE_\pi(\|f\|)=\frac{1}{M}\int_\Omega \|f\|,$$ 
and 
$$\int_{\IE_\pi(\|f\|)>M}\IE_\pi (\|f\|)= \int_{\IE_\pi(\|f\|)>M}\|f\|.$$ 
Hence, $\underset{M\rightarrow +\infty}{\lim}\int_{\IE_\pi(\|f\|)>M}\IE_\pi (\|f\|)=0$ uniformly in $f\in \HH$, provided that $\HH$ is uniformly integrable.

Now, by  Jensen's inequality, since $x\mapsto \|x\|^p$ is convex, $\|\IE_\pi(f)\|^p\leq \IE_\pi(\|f\|^p)$, $\mu$-a.e on $\Omega$.  Therefore, the $p$-uniform integrability of $\HH$ implies the uniform integrability of 
$$\{\|f\|,\|f\|^p,\IE_\pi(\|f\|),\IE_\pi(\|f\|^p),\|\IE_\pi(f)\|^p: f\in \HH,\pi\in \Pi\}.$$  
It results that, there is $\delta\in ]0\;,\varepsilon]$, such that $\mu(E)\leq \delta$ implies  $$\max\left\{\int_E\|f\|,\int_E\|f\|^p,\int_E\IE_\pi(\|f\|),\int_E\IE_\pi(\|f\|^p),\int_E\|\IE_\pi(f)\|^p\right\}\leq \varepsilon, \forall f\in \HH,\forall \pi\in \Pi.$$ 

Using ($\FF$), let $\pi_0\subset \AA^+$ be a partition of $\Omega$, $\Omega_f\in \AA^+,f\in \HH$, such that $\mu(\Omega\setminus \Omega_f)\leq \delta\leq\varepsilon$ and $\EssOsc(f,\Omega_f\cap K)\leq \delta\leq\varepsilon$ for all $f\in \HH$ and $K\in \pi_0$. Clearly, for every $\pi>\pi_0$, $\EssOsc(f,\Omega_f\cap K)\leq \delta\leq\varepsilon$ for all $f\in \HH$ for every $K\in \pi$. Let $\pi>\pi_0$ be fixed and let us denote, for $f\in \HH$, $$\pi_f=\{K\in \pi : \mu(\Omega_f\cap K)>0\}.$$

Then, for all $f\in \HH$, setting $I_\pi=\int_\Omega \|f-\IE_\pi(f)\|^p$, 

$$\begin{array}{rl}
I_\pi&\leq \underset{K\in \pi}{\sum}\int_{\Omega_f\cap K}\|f-m_K(f)\|^p  +2^p\int_{\Omega\setminus \Omega_f}\|f\|^p+\|\IE_\pi(f)\|^p,\\
\\
&\leq 2^p\underset{K\in \pi_f}{\sum}\int_{\Omega_f\cap K}\|f-m_{\Omega_f\cap K} (f)\|^p +2^p\underset{K\in \pi_f}{\sum}\mu(\Omega_f\cap K) \|m_{\Omega_f\cap K}(f)-m_K(f)\|^p+ 2^{p+1}\varepsilon,
\\
\\
&\leq 2^p \mu(\Omega_f).\varepsilon^p +2^p\underset{K\in \pi_f}{\sum}\mu(\Omega_f\cap K)\|(\frac{1}{\mu(\Omega_f\cap K)}-\frac{1}{\mu(K)}) \int_{\Omega_f\cap K}f-\frac{1}{\mu(K)}\int_{K\setminus \Omega_f} f\|^p+2^{p+1}\varepsilon,\\
\\
&\leq 2^p \mu(\Omega_f).\varepsilon^p +2^{2p}\underset{K\in \pi_f}{\sum}\left(\frac{\mu(K\setminus \Omega_f)^p{\mu(\Omega_f\cap K)^{1-p}}}{\mu(K)^p}{\|\int_{\Omega_f\cap K}f\|^p}+\frac{\mu(\Omega_f\cap K)}{\mu(K)^p}\|\int_{K\setminus \Omega_f} f\|^p\right)+2^{p+1}\varepsilon.\\
\end{array}
$$
But, from H\"older's inequality, for every $K\in \pi$, such that $\mu(\Omega_f\cap K)>0$ and $\mu(K\setminus \Omega_f)>0$, 
$$\mu(\Omega_f\cap K)^{1-p}\|\int_{\Omega_f\cap K}f\|^p\leq \int_{\Omega_f\cap K}\|f\|^p\text{ and }\|\int_{K\setminus \Omega_f} f\|^p\leq \mu(K\setminus \Omega_f)^{p-1}\int_{K\setminus \Omega_f}\| f\|^p.$$

Then, 
$$\begin{array}{rl}
I_\pi& \leq 2^p \mu(\Omega_f).\varepsilon^p +2^{2p}\underset{K\in \pi_f,\mu(K\setminus \Omega_f)>0}{\sum}\left(\frac{\mu(K\setminus \Omega_f)^p}{\mu(K)^p}{\int_{\Omega_f\cap K}\|f\|^p}+\right.\\
&\qquad\qquad\qquad\qquad\qquad\qquad\qquad\qquad \qquad\qquad\left.\frac{\mu(\Omega_f\cap K)\mu(K\setminus \Omega_f)^{p-1}}{\mu(K)^p}\int_{K\setminus \Omega_f}\| f\|^p\right)+2^{p+1}\varepsilon,\\
\\
& \leq 2^p \mu(\Omega_f).\varepsilon^p +2^{2p}\underset{K\in \pi_f,\mu(K\setminus \Omega_f)>0}{\sum}\left[\left(\frac{\mu(K\setminus \Omega_f)}{\mu(K)}\right)^{p-1}\frac{\mu(K\setminus \Omega_f)}{\mu(K)}\int_{\Omega_f\cap K}\|f\|^p\right.+\\

&\qquad\qquad\qquad\qquad\qquad\qquad\qquad\qquad \qquad \left.\frac{\max\{\mu(\Omega_f\cap K),\mu(K\setminus \Omega_f)\}^{p}}{\mu(K)^p}\int_{K\setminus \Omega_f}\| f\|^p\right]+2^{p+1}\varepsilon,\\

\\
& \leq 2^p \mu(\Omega_f).\varepsilon^p +2^{2p}\underset{K\in \pi_f}{\sum}\left(\mu(K\setminus \Omega_f)\frac{1}{\mu(K)}\int_{\Omega_f\cap K}\|f\|^p+\int_{K\setminus \Omega_f}\| f\|^p\right)+2^{p+1}\varepsilon,\\
\\
& \leq 2^p \mu(\Omega_f).\varepsilon^p +2^{2p}\underset{K\in \pi_f}{\sum}\left(\mu(K\setminus \Omega_f)\frac{1}{\mu(K)}\int_{K}\|f\|^p+\int_{K\setminus \Omega_f}\| f\|^p\right)+2^{p+1}\varepsilon,\\

\\
&\leq 2^p \mu(\Omega_f).\varepsilon^p +2^{2p}\left(\int_{\Omega\setminus \Omega_f}\IE_\pi(\|f\|^p)+\int_{\Omega\setminus \Omega_f}\| f\|^p\right)+2^{p+1}\varepsilon,\\
\\
&\leq [\mu(\Omega)+2^{p+1}+2]2^p\varepsilon=T\varepsilon.
\end{array}
$$
Where, we used in the last inequality the fact that $\varepsilon \in ]0,1]$ and denoted by $T$ the constant $[\mu(\Omega)+2^{p+1}+2]2^p$ independent from $f$, $\pi$ and $\varepsilon$. Therefore, 
Condition 2) of Lemma \ref{RieszV} follows straightforwardly. 
\end{proof}
\begin{remark}
The proof of the previous Theorem establishes that $p$-uniform integrability and Fr\'echet oscillation restriction condition ($\FF$) implies Condition 2) of Lemma \ref{RieszV}. 

\end{remark}

We end with a comment on tightness and integral tightness. Let us start by observing that integral tightness does not imply tightness. Consider the following~:
\begin{example}Consider the space of absolutely summable sequences $l^1$ endowed with its usual norm and set 
$\Omega=[0,\;1]$ endowed with its Borel $\sigma$-algebra $\AA$ and set $\mu$ to be the Lebesque measure. 
Let $(r_n)_{n\geq 1}$ be the Rademacher sequence on $[0,\;1]$ and $(e_n)_{n\geq 1}$ the canonical basis of $l^1$: $e_n$ is the element of $l^1$, with vanishing terms except the $n^{th}$ one equals to $1$.  Set $\HH=\{r_ne_n :n\geq 1\}\subset L^1(\mu,l^1)$. For every $A\in \AA$, $$\|\int_A r_n e_n d\mu\|=|\int_A r_n d\mu |\|e_n\|=|\int_A r_n d\mu|.$$
 Since, $(r_n)_n$ converges weakly in $L^1(\mu)$ to $0$, the sequence of norms $\|\int_A r_n e_n d\mu\|$ converges to $0$ in $l^1$. Consequently, the set $\{\int_A r_n e_n d\mu: n\geq 1\}$ is relatively norm compact in $l^1$, for every $A\in \AA$. That is $\HH$ is integral tight. However, for $\varepsilon=\frac{1}{4}$, every compact $K$ in $l^1$ cannot contain all the vectors $e_n,n\geq 1$. Hence, for every compact $K$ of $l^1$, there is $e_{n_K}\notin K$. Then, $\mu\{\omega\in \Omega : r_{n_K}e_{n_K}\notin K\}\geq \frac{1}{2}>\varepsilon$. This means that $\HH$ is not tight.

\end{example}

In contrast with the previous example, it is known that tightness together with $p$-uniform integrability imply integral tightness \cite{CAS79,JAL93}.  A straightforward application of this result together with the results above allow to establish a vector valued (separable) version of the well known Vitali's Theorem. 
 
\begin{corollary} Assume that $X$ is separable. A subset $\HH$ of $L^p(\mu,X)$ is relatively norm compact iff it is relatively compact for the convergence in measure topology and $p$-uniformly integrable. 
\end{corollary}

\begin{proof}$\Rightarrow)$ is obvious. 

$\Leftarrow)$ Since $\HH$ is relatively compact in measure, it is tight. Together with its $p$-uniform integrability, this implies its integral tightness \cite{CAS79}. The Fr\'echet oscillation $(\FF)$ results as well from the convergence in measure \cite{SaV95}. We conclude the result from Theorem \ref{VecV} .
\end{proof}

\bibliographystyle{plain}

\begin{thebibliography}{99}
\bibitem[Aubin, 1963]{AUB63}J. P. Aubin, Un th\'eor\`eme de compacit\'e, \emph{C. R. Acad. Sci. Paris} 256 (1963), 5042-5044.

\bibitem[Balder{, }Giraldi and Jalbi, 1994]{BGJ94} E. J. Balder, M. Girardi and V. Jalbi, From weak to strong types of $\LL_E^1$-convergence by the Bocce criterion, \emph{Studia Mathematica} 111(1994), 241-262.

\bibitem[Bogachev, 2007]{BOG07} V. I. Bogachev, Measure theory (2 volumes). Springer-Verlag, Berlin, 2007.

\bibitem[Castaing, 1979]{CAS79} C. Castaing, Un r\'esultat de compacit\'e li\'e \`a la propri\'et\'e des ensembles Dunford-Pettis dans $L^1_F(\Omega, A, \mu)$, \emph{Travaux S\'em. Anal. Convexe} 9(1979), no. 2, Exp. No. 17, 7 pp.

\bibitem[Diaz and Mayoral, 1999] {DiM99} S. Diaz and F. Mayoral, On compactness in spaces of Bochner integrable functions, \emph{Acta Math. Hungar., 83(1999), 231-239.}

\bibitem[D\'iaz and Mayoral, 1998] {DiM98} S. D\'iaz and F. Mayoral, Compactness in measure and the sequential Bourgain property, \emph{Arch. Math.} 71(1998), 55-62.

\bibitem[Diestel and Uhl, 1977] {DiU77} J. Diestel and  J. J. Uhl, Vector measures. Mathematical Surveys, Number 15, American Mathematical Society, 1977.

\bibitem[Jalbi, 1993]{JAL93} V. Jalbi, Contribution aux probl\`emes de convergence des fonctions vectorielles
et des int\'egrales fonctionnelles, \emph{PhD thesis, University  Montpellier II,} 1993. 

\bibitem[Lions, 1961] {LIO61}J. L. Lions, Equations diff\'erentielles op\'erationelles et probl\`emes aux limites, \emph{Springer, Berlin}, 1961.

\bibitem[van Neerven, 2014] {NEE14} J. van Neerven, Compactness in the Lebesgue-Bochner spaces $L^p(\mu,X)$, \emph{Indagationes Mathematicae}, 25 (2014), 389-394. 
\bibitem[van Neerven, 2007] {NEE07} J. van Neerven, Compactness in vector-valued Banach function spaces, \emph{Positivity} 11 (2007) 461-467.
\bibitem[Rossi and Savar\'e, 2003]{RoS03} R. Rossi and G. Savar\'e, Tightness, Integral Equicontinuity and Compactness for Evolution Problems in Banach Spaces, \emph{Ann. Scuola Norm. Sup. Pisa Cl. Sci.}, 5(2003), 395-431.

\bibitem[Saadoune and Valadier, 1995] {SaV95} M. Saadoune and M. Valadier, Convergence in measure. Local formulation of the Fr\'echet criterion, \emph{C. R. Acad. Sci. Paris, S\'erie I,}  320(1995), 423-428.

\bibitem[Simon, 1987] {SIM87} J. Simon, Compact Sets in the space $L^p(0, T ;B)$, \emph{Ann. Mat. Pura Appl.} 146 (1987), 65-96.

\end{thebibliography}

\end{document}